\documentclass{amsart}

\usepackage{amssymb,amsfonts}
\usepackage[all,arc]{xy}

\usepackage[dvipsnames]{xcolor}
\usepackage{hyperref}
\hypersetup{colorlinks=true,citecolor=blue,linkcolor=blue,urlcolor=blue}
\usepackage{mathrsfs}
\usepackage{physics}
\usepackage{tikz}
\usepackage{comment}
\usepackage{verbatim}
\usepackage{caption}
\usepackage{graphicx}
\usepackage{float}
\usepackage{bm}
\usepackage{bbm}
\usepackage{mathtools}
\usepackage{geometry}

\usepackage[shortlabels]{enumitem}

\def\*#1{\mathbf{#1}}

\newcommand\eps{\varepsilon}

\newcommand\RR{\mathbb{R}}

\newcommand\R{\mathbb{R}}
\newcommand\E{\mathbb{E}}
\renewcommand\P{\mathbb{P}}
\newcommand\Id{\mathrm{Id}}
\newcommand\Hess{\mathrm{Hess}}
\newcommand\ds{\displaystyle}

\newtheorem{thm}[equation]{Theorem}
\newtheorem{cor}[equation]{Corollary}
\newtheorem{prop}[equation]{Proposition}
\newtheorem{lem}[equation]{Lemma}

\newtheorem{prob}{Problem}

\theoremstyle{definition}

\theoremstyle{remark}
\newtheorem{rmk}[equation]{Remark}

\numberwithin{equation}{section}

\allowdisplaybreaks

\title{On Talagrand's Convexity Conjecture
}

\author[Dongming (Merrick) Hua]{Dongming (Merrick) Hua}
\author[Antoine Song]{Antoine Song}
\author[Stefan Tudose]{Stefan Tudose}

\address{California Institute of Technology\\ Linde Hall, \#1200 E. California Blvd., Pasadena, CA 91125}
\email{dhua@caltech.edu}
\address{California Institute of Technology\\ 177 Linde Hall, \#1200 E. California Blvd., Pasadena, CA 91125}
\email{aysong@caltech.edu}
\address{Princeton University, Department of Mathematics, Princeton, NJ 08544}
\email{studose@princeton.edu}

\begin{document}

\begin{abstract}

We prove that any random vector in $\R^n$ which is dominated in convex order by a standard Gaussian vector can be written as the sum of three standard Gaussian vectors. This implies that any $1$-subgaussian random vector in $\R^n$ is the sum of a universal number of Gaussian vectors. It also solves M. Talagrand's convexity problem, which in turn implies a weak version of a combinatorial analogue to the problem. 
\end{abstract}
\maketitle

\section{Introduction}
 M. Talagrand posed the following question\footnote{Talagrand asks that the $A$ be balanced, namely that $\lambda x \in A$ for any $x \in A, \lambda \in [-1, 1]$. The second-named author shows in \cite[Theorem 1.1]{Son26} that these are equivalent formulations.} about subsets of $\RR^{n}$ with large Gaussian measure, where $A + B$ denotes the Minkowski sum and $\gamma_n$ is the standard Gaussian measure, see \cite[Problem 2.3]{Tal95}\cite[Conjecture 2.1]{Tal10}\cite[Problem 2.1]{Tal26}\cite[Problem 54]{Green24}:
\begin{prob}[Convexity problem \cite{Tal95,Tal10, Tal26}] \label{T}
Does there exist a positive integer $q$ such that for any $n\geq 1$ and any closed set $A$ in $\R^n$ with $\gamma_n(A) > \frac{2}{3}$, there is a convex body $K$  in $\R^n$ such that
$$\gamma_n(K)\geq \frac{1}{2}\quad \text{and}\quad  K\subset \underbrace{\vphantom{\Big|}A+\cdots+A}_{q\ \mathrm{times}}\quad ?$$
\end{prob} 

Motivations and variants for this problem are discussed at length in \cite{Tal95,Tal10,Tal26} where M. Talagrand identified several fruitful  connections to probability and combinatorics, see also \cite{Joh25,Son26} for further discussions.
Informally, Problem \ref{T} asks whether one can create convexity from large sets in Gaussian spaces through a uniform number of  Minkowski sum operations. What makes such a dimensionless property surprising is that, as is well known, large sets in Gaussian spaces can  look ``small'' geometrically due to concentration of measure. In fact, several strengthenings of this property have been shown to be false. In \cite[Proposition 2.6]{Tal95} M. Talagrand showed that one cannot take $q =2$ even if we allow rescaling $K$ by a universal constant, and in \cite[Theorem 1.3]{Joh25} S. Johnston proved based on optimal transport theory that if one is required to use convex operations instead of the Minkowski sum, the stronger variant of Problem \ref{T} is false. In order to make progress on Problem \ref{T}, the second-named author proposed in a previous work a strategy based on showing that Problem \ref{T} is actually \emph{equivalent} to the following problem concerning subgaussian random vectors, see \cite[Theorem 1.1]{Son26}. Recall that a random vector $X$ in $\R^n$ is centered $\kappa$-subgaussian if $\E[X]=0$ and for any unit vector $v$, we have $\mathbb{P}[|\langle X , v \rangle| \geq t] \leq 2\exp\left( -\frac{t^{2}}{2\kappa^{2}}\right)$.

\begin{prob}[Subgaussian vector problem]\label{S}
Does there exist a positive integer $q'$ such that for any $n\geq1 $ and any centered $1$-subgaussian random vector $X$ in $\R^n$, there exist standard Gaussian random vectors $G_1,...,G_q'$  in $\R^n$ with
$$X  =G_1+ \cdots +G_q'\quad ?$$
\end{prob}

A benefit of reformulating the geometric Problem \ref{T} into a question about sums of Gaussian vectors like Problem \ref{S} is to suggest new lines of attack from analysis, where powerful methods are available. This insight led to some progress on sums of Gaussian vectors and sums of large sets in \cite{Son26}. 
Pushing this approach further, we resolve Problem \ref{S}, and thus Problem \ref{T}, in the affirmative. The main tool is the following structural theorem about random vectors $X$ in $\RR
^{n}$ which are dominated in convex order by a standard Gaussian random vector $G$, meaning that $\mathbb{E}[f(X)] \leq \mathbb{E}[f(G)]$ for any convex function $f \colon \RR^{n} \to \RR$. We write $X \preceq_{\mathrm{cx}} G$ to denote domination in the convex order. 
\begin{thm}[Three Gaussians]\label{thm:main}
Given a random vector $X$ in $\mathbb{R}^{n}$ with $X \preceq_{\mathrm{cx}} G$ for a standard Gaussian random vector $G$, there are three standard Gaussian random vectors $G_1,G_2,G_3$ in $\R^n$ such that
$$X =G_1+G_2+G_3.$$
\end{thm}

In \cite[Corollary 1.2]{Van25}, van Handel showed the following extension of M. Talagrand's celebrated subgaussian comparison theorem  for subgaussian random vectors \cite[Theorem 2.10.11]{Tal21}: there exists a universal constant $\kappa > 0$ so that $\kappa X \preceq_{\mathrm{cx}} G$ for any centered $1$-subgaussian $X$ (see also \cite[Subsection 2.4.2]{Son26} for an alternative derivation of this fact). Thus, Theorem \ref{thm:main} implies the following:
\begin{cor}\label{cor:subgaussian}
There exists $\kappa>0$ such that for any $n\geq 1$ and any centered $1$-subgaussian random $X$ in $\mathbb{R}^n$, there exist three standard Gaussian random vectors $G_1,G_2,G_3$ in $\mathbb{R}^n$ with
$$\kappa X =G_1+G_2+G_3.$$
\end{cor}
A positive answer to Problem \ref{S} is then immediate from Corollary \ref{cor:subgaussian}
since we can then take $q' = 3\lceil \kappa^{-1}\rceil.$ 
We remark that standard examples (e.g. see \cite[Subsection 1.2]{LSS22}) show that sums of two Gaussian vectors are not enough to express all centered $\kappa$-subgaussian random vectors, so the number three in Theorem \ref{thm:main} and Corollary \ref{cor:subgaussian} is sharp. Corollary \ref{cor:subgaussian} generalizes \cite[Theorem 0.3]{Son26}, where the statement was shown in dimension $1$ using an explicit construction. 
It also provides a strengthening of the subgaussian comparison theorem \cite[Theorem 2.10.11]{Tal21}, \cite[Corollary 1.2]{Van25}.
Other relevant results about sums of random vectors include \cite[Theorem 2]{Rhe92}, \cite[Theorem 5]{Mao19}, \cite[Lemma 5]{LSS22}, and \cite{MGV24}.

Once Corollary \ref{cor:subgaussian} is known, by the equivalence between Problem \ref{T} and Problem \ref{S}, we can already conclude that the convexity conjecture, namely Problem \ref{T}, has a positive answer.  
In fact, the following improvement of the first version of our paper, observed by M. Talagrand, yields the following corollary with explicit constants\footnote{See Section \ref{section:proof convexity} for a discussion of the constant $\frac{5}{6}$ versus $\frac{2}{3}$.}:
\begin{cor}\label{cor:main}
For any $n\geq 1$ and any closed set $A$ in $\R^n$ with $\gamma_n(A) > \frac{5}{6}$,
there is a symmetric convex body $K$  in $\R^n$ such that $$\gamma_n(K)\geq \frac{1}{2}\quad \text{and}\quad  K\subset 4(A+A+A).$$
\end{cor}
In the first version of our paper, we established Corollary \ref{cor:main} without explicit constants by using a refinement of \cite[Theorem 1.1]{Son26}.  Shortly after, M. Talagrand observed a shorter and more direct argument based on the Hahn-Banach theorem to deduce Corollary \ref{cor:main} from Theorem \ref{thm:main}. This is the proof which  we will present in Section \ref{section:proof convexity}.
Corollary \ref{cor:main} then settles Problem \ref{T} thanks to standard facts about the geometry of Gaussian spaces, see Section \ref{section:proof convexity}.


 M. Talagrand asked in \cite[Problem 2.3]{Tal26} for a constructive proof  of Corollary \ref{cor:main}. Our proof proceeds by contradiction and is not constructive.
 In \cite[Problem 0.8]{Son26}, the second-named author proposed a Riemannian version of the convexity problem for spaces with nonnegative Ricci curvature.
 \medskip

In retrospect, one of the main takeaways of this paper is that, while Corollary \ref{cor:main} is a geometric statement, its proof and surrounding ideas connect ingredients from a surprisingly broad range of fields, including basic functional analysis and stochastic calculus, discrepancy theory in computer science \cite{Dad19}, Caffarelli's contraction theorem in optimal transport \cite{Caf00}, Talagrand's subgaussian comparison theorem in probability \cite{Van25}, Schr\"{o}dinger bridges in mathematical finance \cite{Kay25, NutzWiesel, schrodingerbridge}, and Laguerre tessellations or power diagrams in materials science \cite{inverselaguerre} and economics \cite{KMSW24}.

\medskip

It is noteworthy that the above results, which are proved using purely analytical methods, also have combinatorial consequences.
In \cite[Problem 3.4]{Tal95} \cite[Conjecture 7.1]{Tal10} \cite[Problem 2.4]{Tal26} M. Talagrand proposed several combinatorial analogues of Problem \ref{T}. 
Given a subset $I$ of $[N] \coloneqq \{1, \ldots, N\}$, we may define $$H_{I} \coloneqq \{J \subseteq [N] \mid I \subseteq J\}.$$ Note that by identifying the subsets of $[N]$ with elements of $\{0, 1\}^{N}$, we may view $H_{I}$ as a subset of $\{0, 1\}^{N}$. For $0 < p < 1$, we define a product measure $\mu_{p}$ on $\{0, 1\}^{N}$ by $((1-p)\delta_{0} + p \delta_{1})^{\otimes N}$.
We say that a subset $S \subseteq \{0, 1\}^{N}$ is $p$-small if there exists a family $\mathcal{I}$ of subsets of $[N]$ such that $$S \subseteq \bigcup_{I \in \mathcal{I}}H_{I}\quad \text{and}\quad
\sum_{I \in \mathcal{I}}\mu_{p}(H_{I}) = \sum_{I \in \mathcal{I}}p^{|I|} \leq \frac{1}{2}.$$ 
Given $A \subseteq \{0, 1\}^{N}$ and an integer $q>0$, we define 
$$A^{(q)} \coloneqq \left\{x \in \{0, 1\}^{N} \mid \forall \; x^{(1)}, \ldots, x^{(q)} \in A,\ \exists i \leq N,\ x_{i} = 1,\, x_{i}^{(1)} = \cdots = x_{i}^{(q)} = 0\right\}.$$ Translated into the language of subsets of $\{1, \ldots, N\}$, $A^{(q)}$ is the collection of subsets of $\{1, \ldots, N\}$ which cannot be covered by $q$ elements of $A$.
A version of the combinatorial conjectures  \cite[Conjecture 7.1]{Tal10} then asks:
\begin{prob}
[\cite{Tal10}]
\label{C}
    Does there exist an integer $q>0$  so that for any $N\geq 1$, any $0 < p <1$, and any subset $A$ of $\{0, 1\}^{N}$ with $\mu_{p}(A) \geq 1 - \frac{1}{q}$, 
    $A^{(q)}$ is $p$-small?
\end{prob}

One motivation for this conjecture and similar variants in \cite{Tal10, Tal26} comes from a reformulation of Problem \ref{T} in terms of half-spaces \cite[Conjecture 3.3]{Tal10}. In the combinatorial analogue, the spaces $H_{I}$ play the role of half-spaces, and the measure $\mu_{p}$ corresponds to the Gaussian measure. As outlined in \cite[Section 3]{Tal95}, an affirmative answer to Problem \ref{T} actually implies an affirmative answer to a weak version of Problem \ref{C}. In fact, as we will explain in Appendix \ref{app: combcons}, a consequence of Corollary \ref{cor:main} is the following:
 \begin{cor}\label{cor:comb}
 There exists $\varepsilon>0$  and $L\geq 1$ so that for any $N\geq 1$, any $0 < p <1$, and any subset $A$ of $\{0, 1\}^{N}$, if $\mu_{p}(A) \geq 1 - \varepsilon$ then $A^{(3)}$ is $p^L$-small.
 \end{cor}
We emphasize that the conclusion of Corollary \ref{cor:comb} is $p^L$-smallness for some universal  $L\geq1$, which is much weaker than the $p$-smallness property asked by Problem \ref{C}. Nevertheless, note that the original version of the combinatorial convexity problem \cite[Problem 3.4 (The main combinatorial problem)]{Tal95} is indeed formulated\footnote{In \cite[Problem 3.4 (The main combinatorial problem)]{Tal95}, the set $A$ is additionally assumed to be ``decreasing'', which means that $y\in A$ whenever $y\subset x$ for some $x\in A$.
This can always be assumed because if $A$ is not decreasing and if $A_0$ is the smallest decreasing subset of $\{0,1\}^N$ containing $A$, then $A^{(q)} = A_0^{(q)}$.} in terms of $p^L$-smallness. In Corollary \ref{cor:comb}, the integer $q$ is just $3$, and it could be that $q=2$ suffices.

It was observed in \cite[Sections 7 and 8]{Tal10}\cite{Tal26} that the stronger combinatorial conjectures are closely connected to the Kahn-Kalai conjecture solved by J. Park and H.T. Pham \cite{Par24a,Pha25}, and a conjecture on selector processes \cite[Conjecture 5.7]{Tal10} proved in \cite[Theorem 1.2]{Par24b}. For instance, it can be checked that \cite{Par24a}
implies that if $\mu_p(A) \ge 1/2$, then $A^{(1)}$ is $\Omega(p/\log(N))$-small.
Corollary \ref{cor:comb} yields a new result when $p$ is larger than $\Omega((\log N)^{-\frac{1}{L-1}})$.
It is unclear if there is a path from the proof of Corollary \ref{cor:comb} to the stronger conjectures.

\subsection{Outline}
Theorem \ref{thm:main} reduces to the case when $X$ is a finitely supported random vector in $\mathbb{R}^n$ dominated by a standard Gaussian in convex order, see Lemma \ref{lem:discrete approx}. A key ingredient in our approach is a strengthening of the classical Strassen's theorem, which tells us that there is a martingale coupling between $X$ and $G$. By maximizing a certain notion of entropy in the choice of $G$, we can choose the martingale coupling so that the conditional distribution of $G$ given $X$ is $1$-uniformly log-concave. This is proved in  Proposition \ref{log-concave}. Then by \cite[Section 2]{Son26}, $G-X$ can be expressed as the sum of two standard Gaussian vectors. Consequently, $X$ is indeed a sum of three Gaussian vectors, which gives Theorem \ref{thm:main}. As explained above, Corollary \ref{cor:main} follows. 

In Appendix \ref{app: combcons}, we explain the argument to obtain the combinatorial consequence, Corollary \ref{cor:comb}, which is due to M. Talagrand \cite{Tal95}. 
First, Corollary \ref{cor:main}  implies a similar statement for ``solid'' sets\footnote{A set  $X\subset \R^n$ is solid if $(y_1,...,y_n)\in X$ whenever $(x_1,...,x_n)\in X$ and $|y_i|\leq |x_i|$ for all $i=1,...,n$.}. 
Thanks to results on Gaussian measures in Banach spaces, Corollary \ref{cor:comb} is then obtained by applying the solid version of Corollary \ref{cor:main} to the solid set $F_t^{-1}(A)\subset \R^N$, where $t$ is such that $p=\gamma_1(|x|\geq t)$ and $F_t((x_1,...,x_N)) := (1_{|x_1|\geq t},...,1_{|x_N|\geq t})$. 

In Appendix \ref{section:laguerre}, we discuss a result, Proposition \ref{prop:tessellation}, which was generated by a large language model and which provides an alternative approach to one step of the proof of Theorem \ref{thm:main}. We explain that it is in fact a corollary of a theorem about Laguerre tessellations \cite[Theorem 3.5]{inverselaguerre}, which itself follows from our Proposition \ref{log-concave}. This last observation provides a simple  connection between the literature on  Laguerre tessellations and Schr\"{o}dinger bridges.

\subsection{Statement on A.I. use}\label{sec:aistatement}

The content of this paper is the product of human authorship, apart from the explicitly indicated portions of Appendix \ref{section:laguerre}, and routine typographical error detection and literature searches. The first and second-named authors, working independently from the third-named author, had reached a resolution of Problem \ref{S} based on a result whose proof was generated  by GPT-5.5 Pro during a conversation\footnote{The interested reader can find a transcript here. {https://chatgpt.com/share/69fae923-68f8-83e8-9dea-aceba3639524}.} initiated by the first-named author, see Appendix \ref{section:laguerre}. In parallel, the third-named author independently arrived at a complete proof as well. Upon comparing the two initial approaches, we found the third-named author’s approach to be more general and conceptual. As such, we decided to focus and expand upon only the latter proof of Problem \ref{S} in the main body of the paper.

\subsection{Acknowledgements}
We are indebted to Ramon van Handel for putting the authors in contact,  which led to the present collaboration, and also for useful exchanges, references, and a minor correction. We are  grateful to Michel Talagrand for inspiring discussions, for letting us incorporate his simplification of one part of our paper and for a minor correction.
We also thank Jinyoung Park and Huy Tuan Pham for their explanations on the notion of $p$-smallness in combinatorics. 
 A.S. would like to thank people who helped with \cite{Son26}, which played a significant role in the genesis of this paper: this includes Assaf Naor, Samuel Johnston, Shiqi Song, Boaz Klartag, Yair Shenfeld, Daniel Dadush, Alexandar Nikolov, Haotian Jiang, Roman Vershynin, Jinyoung Park and Tom Hutchcroft.
A.S. is grateful to his mother-in-law Fengzhen Liu and wife Amy Zheng for their invaluable help with the newborn and for tolerating him working on this paper at the hospital.
A. S. was partially supported by NSF grant DMS-2405175 and an Alfred P. Sloan Research Fellowship. D. H. was partially supported by the Danny Koh Graduate Fellowship at Caltech. 

\medskip

\section{Basic properties of sums of Gaussian vectors}\label{sec: sums}

\subsection{Notations}

We denote the identity matrix in $\R^n$ by $\Id$ or $I_n$. 
Given random vectors $X$ in $\R^n$ and $Y_1,...,Y_k$ in $\R^n$ such that the $Y_i$'s are defined on the same probability space, we often abuse notation and  write $X=Y_1+\cdots+Y_k$ when $X$ and $Y_1+ \cdots +Y_k$ have the same probability distribution. 
We say that $X$ is the sum of $k$ standard Gaussian vectors in $\R^n$ 
if there is a coupling of $k$ random vectors  $Y_1,...,Y_k\sim \mathcal{N}(0,I_n)$ such that the sum has same probability distribution as $X$.  
Given a random vector $X$ we may sometimes define another random vector $Y$ with respect to $X$ and consider $X+Y$ under the implicit assumption that both vectors are defined on a common probability space (which may be enlarged if necessary).

\subsection{Gaussian, subgaussian, and log-concave vectors}

Given $\kappa>0$, a random vector $Y$ in $\R^n$ is called $\kappa$-subgaussian \cite[Definition 3.4.1, Proposition 2.6.1]{Ver} if for any unit vector $v$, we have $\mathbb{P}[|\langle Y,v\rangle |\geq t] \leq 2\exp(-\frac{t^2}{2\kappa^2})$. 

A function $f:\R^n\to \R\cup \{\infty\}$ is called \emph{log-concave} (resp. \emph{$1$-uniformly log-concave}) if $f(x)=\exp(-V(x))$ where $V:\R^n \to \R\cup \{\infty\}$ is convex, $V$ is twice differentiable on the interior of  $\{V<\infty\}$ and $\Hess V\geq 0$ (resp. $\Hess V\geq I_d$) on $\{V<\infty\}$. 
A  random vector in $\R^n$ or its corresponding probability measure on $\R^n$ are called \emph{log-concave} (resp. \emph{$1$-uniformly log-concave}) if the probability measure has a density function which is \emph{log-concave} (resp. \emph{$1$-uniformly log-concave}).

Let us also record here the following standard application of Prokhorov's theorem which will be used later:
\begin{lem}\label{lem:closure}   
Let $q$ be a positive integer. The set of random vectors in $\R^n$ which can be written as the sum of $q$ standard Gaussian vectors is closed under weak limits. 
\end{lem}

\subsection{Convex Order}

    Given two random vectors $X, Y$ in $\RR^{n}$, we say that $X$ is dominated by $Y$ in convex order if $\mathbb{E}[f(X)] \leq \mathbb{E}[f(Y)]$ for any convex $f \colon \RR^{n} \to \RR$, and we write $X \preceq_{\mathrm{cx}} Y$. Similarly, for probability measures $\mu, \nu$ on $\RR^{n}$, we say $\mu \preceq_{\mathrm{cx}} \nu$ whenever $X \preceq_{\mathrm{cx}} Y$ for $X \sim \mu$, $Y \sim \nu$.

Observe that by testing coordinate projections and their negatives, one can deduce that $X \preceq_{\mathrm{cx}} Y$ implies $\mathbb{E}[X] = \mathbb{E}[Y]$. The following result of van Handel shows that there is a universal constant $c > 0$ so that any centered, $c$-subgaussian random vector is dominated in convex order by a standard Gaussian random vector.

\begin{thm}[\cite{Van25}]\label{subgaussian thm}
    There exists a universal constant $c>0$ such that for any centered $1$-subgaussian random vector $X$ in $\R^n$, $cX$  is dominated in the convex order by $G$, where $G\sim \mathcal{N}(0,I_n)$ is a standard Gaussian vector.
\end{thm}

To give some context, the majorizing measure theorem of Talagrand \cite{Tal87} \cite[Theorem 2.10.1]{Tal21} implies the following subgaussian comparison theorem \cite[Theorem 2.10.11]{Tal21}: for any centered $1$-subgaussian random vector in $\R^n$, for some universal $c>0$, 
$\E[f(cX)]\leq \E[f(G)]$ where $G\sim \mathcal{N}(0,I_n)$ and $f$ is any $1$-homogeneous convex function. An alternative approach to this subgaussian comparison theorem due to J. Liu \cite{Liu25} was shown in \cite{Van25} to imply the stronger Theorem \ref{subgaussian thm}.  See also \cite[Subsection 2.4.2]{Son26} for an alternative explanation of Theorem \ref{subgaussian thm}.

\subsection{Lipschitz image lemma}
Here is a simple yet helpful lemma  observed in \cite[Lemma 2.4]{Son26}:
\begin{lem}[{\cite{Son26}}] \label{lem:average of gaussians}
Let $\Psi:\R^n\to \R^n$ be a Lipschitz map with Lipschitz constant at most $C_{\mathrm{Lip}}>0$, and let $G$ be a standard Gaussian random vector in $\R^n$.
Then there are two standard Gaussian random vectors $X,Y$ in $\R^n$ such that 
 $$\Psi(G)-\E[\Psi(G)]=C_{\mathrm{Lip}}\frac{X+Y}{2}.$$

\end{lem}
 \begin{proof}
 For the reader's convenience, we quickly reproduce the proof here.
It relies on  It\^{o}'s formula in stochastic calculus  \cite[Theorem 3.3]{Rev99}.
By rescaling $\Psi$ by a factor $\frac{1}{C_{\mathrm{Lip}}}$, it is enough to show the statement when $\Psi$ is 1-Lipschitz.
Let $\{B_t\}$ be the standard Brownian motion on $\R^n$ starting at $0$. Note that $\Psi(G)-\E[\Psi(G)]$ has same probability distribution as $\Psi(B_1)-\E[\Psi(B_1)].$
By It\^{o}'s formula  \cite[Theorem 3.3]{Rev99}, we have 
$$\Psi(B_1) -\E[\Psi(B_1)] = \int_0^1 \nabla(P_{1-t}\Psi)(B_t) dB_t$$
Here, $\{P_t\}_{t\geq 0}$ is the heat semigroup, and if $\Psi$ is the vector-valued map $(\Psi_1,...,\Psi_n)$, then $P_{1-t}\Psi$ denotes the vector-valued map $(P_{1-t}\Psi_1,...,P_{1-t}\Psi_n)$ and $\nabla(P_{1-t}\Psi)(x)$ is the linear map whose matrix has $i$-th row equal to $(\frac{\partial}{\partial x_1} P_{1-t}\Psi_i(x), ..., \frac{\partial}{\partial x_n} P_{1-t}\Psi_i(x))$.
Since $\Psi$ is 1-Lipschitz, $P_{1-t}\Psi$ is also 1-Lipschitz and so 
$\|\nabla(P_{1-t}\Psi)\| \leq 1.$
This means by  \cite[Proposition 2.13]{Rev99} that the random variable $\Psi(G)-\E[\Psi(G)]$ is the limit in probability of random variables of the following form:
$$\sum_{j=1}^N A_j(G_{N,j})$$
where $G_{N,j}\sim \mathcal{N}(0,\frac{1}{N}I_n)$ are independent and for each $j\geq 1$, $A_j:\mathbb{R}^n\to \mathbb{R}^n$ is a certain random linear operator with operator norm at most 1 depending only on $G_{N,1},...,G_{N,j-1}$.
It is an exercise (see \cite[Lemma 2.1]{Son26}) to check that if $A':\mathbb{R}^n\to \mathbb{R}^n$ is linear with operator norm at most $1$ and $G'\sim \mathcal{N}(0,I_n)$ then $A'(G')$ is the average of two non-independent standard Gaussian vectors. Thus, for each $j=1,...,N$, conditioned on $A_j$ (which only depends on $G_{N,k}$ where $k=1,...,j-1$),
$$A_j(G_{N,j}) = \frac{X_{N,j}+Y_{N,j}}{2}$$
where $X_{N,j}, Y_{N,j} \sim \mathcal{N}(0,\frac{1}{N}I_n)$. 
Since the latter holds no matter what the values of $G_{N,k}$ are for $k=1,...,j-1$, this gives well-defined random vectors $X_{N,1},...,X_{N,N}\sim \mathcal{N}(0,\frac{1}{N}I_n)$ which are independent, and similarly random vectors $Y_{N,1},...,Y_{N,N}\sim \mathcal{N}(0,\frac{1}{N}I_n)$ which are independent.
Summing in $j$, we get 
$$\sum_{j=1}^N A_j(G_{N,j}) = \frac{X_N+Y_N}{2}$$
where $X_N:=\sum_{j=1}^N X_{N,j}$ and $Y_N:=\sum_{j=1}^N Y_{N,j}$
satisfy $X_N,Y_N \sim \mathcal{N}(0,I_n)$.
By Lemma \ref{lem:closure}, since $\Psi(G)-\E[\Psi(G)]$ is the limit in distribution of $\sum_{j=1}^N A_j(G_{N,j})$, $\Psi(G)-\E[\Psi(G)]\sim \frac{X+Y}{2}$
where $X,Y \sim \mathcal{N}(0,I_n)$.

 \end{proof}

Let us recall Caffarelli's contraction theorem \cite[Theorem 11]{Caf00}, which states that any probability measure on $\mathbb{R}^n$ which is more log-concave than $\gamma_n$ is the pushforward of $\gamma_n$ by a distance non-increasing map (see also \cite{Kol11} for a survey):
\begin{thm}[\cite{Caf00}]\label{caf}
Let $\mu$ be a probability measure on $\R^n$ which is $1$-uniformly log-concave. Then there is a $1$-Lipschitz map $F:\R^n\to \R^n$ such that $F_*(\gamma_n) = \mu$.
\end{thm}

Note that if a random vector $W$ is $1$-uniformly log-concave, then the same holds for $\frac{W}{2}$. Thus, combining the two previous results, we obtain:
\begin{cor} \label{cor:quick app}
If a random vector $W$ in $\R^n$ is $1$-uniformly log-concave, then $$W - \E[W]= G_1+G_2$$ for some standard Gaussian vectors $G_1,G_2 \sim \mathcal{N}(0,I_n)$. 
\end{cor}

This corollary can alternatively be deduced by the more intricate arguments of \cite{Eld16}, which are written for dimension $1$ but should generalize to all dimensions.

\section{Gaussian convex domination implies sum of standard Gaussians}

In this section, we explain how to show Theorem \ref{thm:main} up to Proposition \ref{log-concave}, whose proof will be treated in Section \ref{entropy}. To prove Theorem \ref{thm:main}, it suffices to consider the case when the random vector $X$ takes finitely many values due to Lemma \ref{lem:closure} and the following:

\begin{lem}\label{lem:discrete approx}
Let $G\sim \mathcal{N}(0,I_n)$. 
Any random vector $X$ in $\R^n$ such that $X \preceq_{\mathrm{cx}} G$ is the weak limit of a sequence of finitely supported random vectors $X_k$ in $\R^n$ satisfying $X_k\preceq_{\text{cx}}G$.
\end{lem}

\begin{proof} Consider a series of disjoint partitions $\mathcal{P}_k$ of $\R^n$ each consisting of finitely many Borel measurable sets such that $\mathcal{P}_{k+1}$ is finer than $\mathcal{P}_k$ and $\sigma\left(\cup_{k\ge 1}\mathcal{P}_k \right )$ contains all the Borel sets in $\R^n$. Define $ \mathcal{F}_k=X^{-1}(\sigma(\mathcal{P}_k))$ and the random variable
$X_k=\E[X|\mathcal{F}_k]$. Observe that $X_k$ converges weakly to $X$ (here we use that $X$ has finite first moments since it is dominated in convex order by a Gaussian vector), takes on finitely many values and  is dominated in the convex order by $G$. To see the latter, let $f:\R^n\rightarrow \R$ be a convex function. Then 
$$f(X_k)=f(\E[X|\mathcal{F}_k])\le \E[f(X)|\mathcal{F}_k]$$ 
and by taking the total expectation we get $\E[f(X_k)]\le \E[f(X)]\le \E[f(G)]$.
\end{proof}

Next, recall that a classical theorem of Strassen gives an equivalent characterization of convex domination:
\begin{thm}[\cite{strassen}]
   The probability measure $\mu$ is dominated in the convex order by $\nu$ if and only if there exists a martingale coupling $(X,Y)$ of $\mu$ and $\nu$, i.e. there exist $X\sim \mu$ and $Y\sim \nu$ such that $\E[Y|X]=X$. 
\end{thm}
Strassen's theorem is an abstract existence result: if a probability distribution is dominated in the convex order by another, then one can find a martingale coupling. We will prove in the next section the following stronger structural result. 

\begin{prop} \label{log-concave}
    Let $\mu=\sum\limits_{i=1}^{m} p_i \delta_{x_i}$ be a finitely supported centered probability measure on $\R^n$, and let $\nu$ be a centered  probability measure  on $\R^n$ absolutely continuous with respect to the Lebesgue measure, and with finite first moments. Assume that $\mu\preceq_{\text{cx}}\nu$.
    Then there exists a natural bijection between the set of martingale couplings $(X,Y)$ of $\mu$ and $\nu$, and the collection $\mathcal{C}$ of tuples of nonnegative functions $(f_1,\ldots, f_m)$ (defined up to $\nu$-null sets) where the $f_{i} \colon \RR^{n} \to \RR$ satisfy $$ \sum\limits_{i=1}^{m}p_if_i(x)=1,\ \E\left[f_i(Y)\right]=1,\ \E \left[Yf_i(Y)\right]=x_i.$$ 

 Moreover, we have $X\preceq_{\text{cx}}(1-\varepsilon) Y$ for some $\varepsilon>0$ if and only if there exist $(f_1^*,\ldots, f_m^*)$ in $\mathcal{C}$ such that each $f_i^*$ is of the form
 $$f_i^*(x)=\dfrac{e^{U_i+\langle V_i,x\rangle}}{\sum\limits_{j=1}^{m} p_j e^{U_j+\langle V_j,x\rangle}}$$
 for some $U_i\in \R,\ V_i\in \R^n$. Note that each $f_i^*$ is log-concave. 
\end{prop}

The first part of the proposition is easy to see. Given $\E[Y|X]=X$, we can take
    $$f_i(y)=\frac{\P(X=x_i|Y=y)}{\P(X=x_i)}$$
    Reciprocally, given $(f_1,\ldots, f_m)$ in $\mathcal{C}$, we can take 
    $$Y=\sum\limits_{i=1}^{m} X_i 1_{X=x_i}$$
    where each $X_i$ is independent of $X$ and has density $f_i d\nu$. Note that the expectation of $X_i$ is $x_i$. We save the proof of the second part of the proposition for Section \ref{entropy}. Now we apply this proposition to prove Theorem \ref{thm:main}.

\begin{proof}[Proof of Theorem \ref{thm:main}] Let $X$ be a centered random vector in $\R^n$ dominated in the convex order by a standard Gaussian. By Lemma \ref{lem:discrete approx}, we can assume that $X$ takes finitely many values. If we show that $(1-\varepsilon) X$ is the sum of three standard Gaussian vectors for any $\varepsilon>0$, by Lemma \ref{lem:closure} it follows that $X$ itself can be written as the sum of three standard Gaussian vectors. We can therefore assume that $X\preceq_{\text{cx}}(1-\varepsilon) G$.  By Proposition \ref{log-concave}, we can find $(f_1^*,\ldots, f_m^*)$ in $\mathcal{C}$ with each $f_i^*$ log-concave. We take $G=\sum\limits_{i=1}^{m} X_i 1_{X=x_i}$ where $X_i$ is independent of $X$ with density $f_i^*\gamma_n$ and $X_i$ is $1$-uniformly log-concave. In other words, Proposition \ref{log-concave} produces a martingale coupling $(X,G)$ such that $G$ conditioned on $X=x_i$ has $1$-uniformly log-concave distribution for each $x_i$. By Corollary \ref{cor:quick app}, we know that $X_i-x_i$ can be written as $Y_i+Z_i$ with $Y_i$ and $Z_i$ standard Gaussian vectors. Since $Y_i$ and $Z_i$ are independent of $X$, it follows that the random variables 
$$Y=\sum\limits_{i=1}^{m}Y_{i}1_{X=x_i},\ Z=\sum\limits_{i=1}^{m}Z_{i}1_{X=x_i}$$ are standard Gaussian vectors. We have shown that $G=X+Y+Z$ so $X$ can be written as the sum of three standard Gaussian vectors.\end{proof}

\begin{rmk} \label{rmk: bbbl}
After the first version of this paper was announced, we found out that Kay recently showed a statement very similar to Proposition \ref{log-concave} in the context of a mathematical finance problem in \cite[Section 2 and 3]{Kay25}. While his proof uses a continuity argument\footnote{There is a minor issue in the current version of \cite{Kay25}: a diffeomorphism from $\mathbb{R}^d$ onto a bounded  open  set might not extend continuously to the sphere at infinity. In particular, in the notations of \cite[Section 2]{Kay25}, an element of $\partial F(\mathbb{R}^{N(n+1)})$ may not be of the form \cite[(3.36)]{Kay25}.}, ours is based on a direct variational argument.  
Furthermore, we observe that closely related ideas also appeared in \cite{NutzWiesel} and were later generalized by \cite{schrodingerbridge} to higher dimensions in the context of martingale Schrödinger bridges. The fact that that we can find $(f_1^*,\ldots, f_m^*)\in \mathcal{C}$ such that each $f_i^*$ is log-concave follows formally from \cite[Equation (1.1)]{schrodingerbridge}. However, to apply the results in \cite{schrodingerbridge}, one needs certain technical assumptions on the distributions. We believe that one of these assumptions is the analogue of the fact that the function $g$ defined above has a unique maximum, which is not straightforward to show. 
 In our case, to obtain that, we will use the extra slack condition $X\preceq_{\text{cx}} (1-\varepsilon) Y$. 

\end{rmk}

\section{Proof of Proposition \ref{log-concave}}
\label{entropy}
Let us explain the proof of the second part of Proposition \ref{log-concave}. 
\subsection{Slack in convex domination implies softmax densities.}
Suppose $X \preceq_{\mathrm{cx}} (1-\eps)Y$ for some $\eps > 0$. We wish to find conditional densities $(f_{1}^{*}, \ldots, f_{m}^{*}) \in \mathcal{C}$ of the form discussed in the second part of Proposition \ref{log-concave}. The principle of maximum entropy guides our search and suggests that we should maximize the conditional entropy of $X$ given $Y$, or equivalently minimize the mutual information of $X$ and $Y$. This amounts to maximizing
$$H(f)=\E\left [-\sum\limits_{i=1}^{m}p_if_i(Y)\log f_i(Y)\right]$$
under the constraints given by $(f_1,\ldots, f_m)\in \mathcal{C}$. Writing down the associated Lagrangian, one sees that the extremizers must be of the desired form
$$f_i^{*}(x)=\frac{e^{U_i+\langle V_i,x\rangle}}{\sum\limits_{j=1}^{m}p_j e^{U_j+\langle V_j,x\rangle }}$$
for some $U_i\in \R,\ V_i\in \R^n$.

Informed by this, we are looking for functions $f_i^*$ of the above form satisfying the constraints of $\mathcal{C}$. To find these functions, we consider the mapping $g:\underbrace{\R\times \ldots \times \R}_{m\ \text{times}}\times \underbrace{\R^n\times \ldots \times \R^n}_{m\ \text{times}}\rightarrow \R$ given by 
\begin{align}\label{def of g}
\begin{split}
g(U,V)& = g(U_1,\ldots, U_m,V_1,\ldots, V_m)\\
&=\sum\limits_{i=1}^{m}p_i\left (U_i+\langle V_i, x_i\rangle\right ) -\E \left [\log \left( \sum\limits_{i=1}^{m}p_ie^{U_i+\langle V_i, Y\rangle}\right ) \right ] 
\end{split}
\end{align}
A critical point of $g$ will lead to values of $U_i,V_i$ which make the functions $f_i^*$ defined above satisfy the constraints of $\mathcal{C}$. 

Since $g$ is invariant under translating all the $U_i$ by the same constant and translating all the $V_i$ by the same vector, we mod out the gauge by working on the subspace
\begin{align}\label{def of S}
\mathcal{S}=\left \{(U,V):\ \sum\limits_{i=1}^{m}p_iU_i=0,\ \sum\limits_{i=1}^{m} p_iV_i=0 \right  \}
\end{align}

The function $g:\mathcal{S}\rightarrow \R$ is strictly concave. We are going to show that
$$\lim_{(U,V)\in \mathcal{S},\ \norm{(U,V)}\rightarrow \infty} g(U, V)=-\infty,$$
which implies that $g$ has a unique global maximum on $\mathcal{S}$. 

To that end, we are going to use the assumption $X\preceq_{\text{cx}} (1-\varepsilon)Y$. As explained in the paragraph following Proposition \ref{log-concave}, from the coupling of $X$ and $(1-\varepsilon)Y$ we can find nonnegative functions $\tilde{f}_1,\ldots, \tilde{f}_m:\R^n\rightarrow \R$ with 
$$\sum\limits_{i=1}^{m}p_i\tilde{f}_i(x)=1,\ \E  [\tilde{f}_i \left ((1-\varepsilon)Y\right )  ] =1,\ \E  [(1-\varepsilon) Y\tilde{f}_i \left ((1-\varepsilon)Y\right )  ] =x_i.$$
If we define $f_i:\R^n\rightarrow \R$ by $f_i(x)=(1-\varepsilon) \tilde{f}_i((1-\varepsilon)x) +\varepsilon$, then 
$$\sum\limits_{i=1}^{m}p_if_i(x)=1, \ \E [f_i(Y)]=1,\ \E[ Yf_i(Y)]=x_i$$
and moreover $f_i\ge \varepsilon$. This allows us to write

\begin{align*}
g(U,V)&=	\sum\limits_{i=1}^{m}p_i\left (U_i+\langle V_i, x_i\rangle\right ) -\E \left [\log \left( \sum\limits_{i=1}^{m}p_ie^{U_i+\langle V_i, Y\rangle }\right )  \right ] \\
&=\E \left [\sum\limits_{i=1}^{m}p_i\left (U_i+\langle V_i, Y\rangle- \log \left( \sum\limits_{j=1}^{m}p_je^{U_j+\langle V_j, Y\rangle }\right )\right )f_i(Y) \right ] \\
&= \E \left [\sum\limits_{i=1}^{m}p_i \log \left(\dfrac{e^{U_i+\langle V_i, Y\rangle }}{ \sum\limits_{j=1}^{m}p_je^{U_j+\langle V_j, Y\rangle }}\right )f_i(Y) \right ]\\
&= \E \left [\sum\limits_{i=1}^{m}p_i \log \left(\dfrac{p_ie^{U_i+\langle V_i, Y\rangle}}{ \sum\limits_{j=1}^{m}p_je^{U_j+\langle V_j, Y\rangle}}\right )f_i(Y) \right ]-\sum\limits_{i=1}^{m} p_i\log p_i,
\end{align*}
where we have used $\sum_{i=1}^{m}p_{i}f_{i}(Y)= 1,\  \mathbb{E}[Yf_{i}(Y)] = x_{i}$ in the second equality, and $\mathbb{E}[f_{i}(Y)] = 1$ in the fourth. Since $f_{i} \geq \eps$ and the logarithm terms are negative, we therefore have 
\begin{align*}
g(U, V) \le \varepsilon \E \left [\sum\limits_{i=1}^{m}p_i \log \left(\dfrac{p_ie^{U_i+\langle V_i, Y\rangle }}{ \sum\limits_{j=1}^{m}p_je^{U_j+\langle V_j, Y\rangle}}\right ) \right ]-\sum\limits_{i=1}^{m}p_i \log p_i.
\end{align*}

Expanding the logarithm and using that $\sum\limits_{i=1}^{m}p_i=1,\ \sum\limits_{i=1}^{m} p_iU_i=0$ and $\sum\limits_{i=1}^{m}p_iV_i=0$, we obtain
$$g(U,V)\le -\varepsilon \E \left [ \log \left( \sum\limits_{i=1}^{m}p_ie^{U_i+\langle V_{i}, Y \rangle}\right )\right ]-(1-\varepsilon)\sum\limits_{i=1}^{m}p_i\log p_i.$$
Note that 
\begin{align*}
	\E \left [ \log \left( \sum\limits_{i=1}^{m}p_ie^{U_i+\langle V_i, Y\rangle }\right )\right ]&\ge \E \left [\log \max_{1\le i\le m}(p_ie^{U_i+\langle V_i, Y\rangle }) \right]\\
	&=\E \left [\max_{1\le i\le m} \left (\log p_i+U_i+\langle V_i, Y\rangle \right ) \right ]\\
	&\ge \min_{1\le i\le m}\log p_i+\E \left [\max_{1\le i\le m}U_i+\langle V_i, Y\rangle \right] 
\end{align*}

To finish the proof, we observe that the second term above is homogenous, so we can factor out the scale  $\lambda$:
$$\lim_{\|(U,V)\|\to \infty} \mathbb{E} \left[ \max_{1\le i\le m} (U_i + \langle V_i, Y \rangle) \right] = \lim_{\lambda \to \infty,  \|(U,V)\|=1} \lambda \mathbb{E} \left[ \max_{1\le i\le m} (U_i + \langle V_i, Y \rangle) \right] = \infty.$$
The last equality holds because for any fixed $(U, V) \neq (0, 0)$ we have $\sum_{i=1}^{m} p_i (U_i + \langle V_i, x \rangle) = 0$ and thus $\mathbb{E} \left[ \ds\max_{1\le i\le m} U_i + \langle V_i, Y \rangle \right] > 0$ (or else $Y$ would be supported within a hyperplane, which is impossible); by compactness of the sphere we also have 
$$ \min_{\|(U,V)\|=1} \mathbb{E} \left[ \max_{1\le i\le m} (U_i + \langle V_i, Y \rangle) \right] > 0.$$

We have thereby shown that $$\ds \lim_{(U,V)\in \mathcal{S},\ \norm{(U,V)}\rightarrow \infty} g(U,V)=-\infty.$$ The map $g:\mathcal{S}\rightarrow \R$ has a unique maximum on $\mathcal{S}$ at some point $(\tilde{U},\tilde{V})$. By definition of $\mathcal{S}$ and since $g$ is invariant under uniform translation of $U$ and $V$, we conclude that $(\tilde{U},\tilde{V})$ is in fact a global maximum of $g$ on $\RR^{m} \times (\RR^{n})^{m}$. 

Computing gradients (where differentiation within the expectation is justified by the finite first moment assumption), we see that 
\begin{align*}
    \nabla_{U_{i}} g(U, V) &= p_{i} - \mathbb{E}\left[\frac{p_{i}e^{U_{i}+\langle V_{i}, Y \rangle}}{\sum_{j=1}^{m}p_{j}e^{U_{j}+\langle V_{j}, Y \rangle}}\right],\\
    \nabla_{V_{i}}g(U, V) &= p_{i}x_{i} - \mathbb{E}\left[Y \frac{p_{i}e^{U_{i}+\langle V_{i}, Y \rangle}}{\sum_{j=1}^{m}p_{j}e^{U_{j}+\langle V_{j}, Y \rangle}}\right].
\end{align*}

Since $(\tilde{U}, \tilde{V})$ is a global maximum, the equations $\nabla_U g(\tilde{U},\tilde{V})=0,\ \nabla_{V}g(\tilde{U},\tilde{V})=0$ tell us that if we define
$$f_i^{*}(x)=\frac{e^{\tilde{U}_i+\langle \tilde{V}_i,x\rangle}}{\sum\limits_{j=1}^{m}p_j e^{\tilde{U}_j+\langle \tilde{V}_j,x\rangle }},$$
then $\E[f_i^*(Y)]=1,\ \E [Yf_i^*(Y)]=x_i$. Moreover, $\sum\limits_{i=1}^{m}p_if_i^*(x)=1$ and $f_i^*$ is log-concave. This finishes one direction of the second part of Proposition \ref{log-concave}. 

\subsection{Softmax densities implies slack in convex domination.}

Now suppose that we have log-concave conditional densities of the form
$$f_{i}^{*}(x) = \dfrac{e^{U_i+\langle V_i,x\rangle}}{\sum\limits_{j=1}^{m} p_j e^{U_j+\langle V_j,x\rangle}}$$ so that $(f_{1}^{*}, \ldots, f_{m}^{*}) \in \mathcal{C}$. We want to show the existence of some $\eps' > 0$ so that $X \preceq_{\mathrm{cx}} (1-\eps')Y$. Note that this is equivalent to $(1+\eps)X \preceq_{\mathrm{cx}} Y$ for some $\eps > 0$, so we will show the latter.

By the first part of Proposition \ref{log-concave}, it suffices to construct nonnegative $f_{1, \eps}, \ldots, f_{m, \eps} \colon \RR^{n} \to \RR$ so that 
$$\sum_{i=1}^{m}p_{i}f_{i, \eps}(x) = 1, \ \mathbb{E}[f_{i, \eps}(Y)] = 1, \ \mathbb{E}[Y f_{i, \eps}(Y)] = (1+\eps)x_{i}.$$ For now, let us assume there exists some compact set $K \subset \RR^{n}$ with nonempty interior so that $P(Y \in K) \neq 0$, $\mathbb{E}[Y1_{Y \in K}] = 0$. Then $\mathbb{E}[YY^{T}1_{Y \in K}]$ is the covariance matrix of $Y1_{Y \in K}$, which must be invertible (or else $Y1_{Y \in K}$ would be supported within a hyperplane, which is impossible by our assumptions on $K$ and $Y$).

Using this invertibility, there exist $y_{1}, \ldots, y_{m} \in \RR^{n}$ satisfying $\mathbb{E}[YY^{T}1_{Y \in K}]y_{i} =  x_{i}$. Note that $\sum_{i=1}^{m}p_{i}y_{i}$ is therefore in the kernel of $\mathbb{E}[YY^{T}1_{Y \in K}]$, hence $\sum_{i=1}^{m}p_{i}y_{i} = 0$. Now define $h_{1}, \ldots, h_{m} \colon \RR^{n} \to \RR$ by 
$$h_{i}(z) = \langle z,  y_{i}\rangle 1_{z \in K}.$$ By the previous calculations, we have 
$$\sum_{i=1}^{m}p_{i}h_{i} = 0, \ \mathbb{E}[h_{i}(Y)] = 0, \ \mathbb{E}[Yh_{i}(Y)] = \mathbb{E}[YY^{T}1_{Y \in K}] y_{i} =x_{i}.$$ So, for any $\eps > 0$, defining $f_{i, \eps} = f_{i}^{*} + \eps h_{i}$, we see that 
$$\sum_{i=1}^{m}p_{i}f_{i, \eps}(x) = 1, \ \mathbb{E}[f_{i, \eps}(Y)] = 1, \ \mathbb{E}[Y f_{i, \eps}(Y)] = (1+\eps)x_{i}.$$ Since each $f_{i}^{*}$ is continuous and strictly positive, and since each $h_{i}$ vanishes outside of the compact set $K$, by choosing $\eps$ sufficiently small we can ensure that the $f_{i, \eps}$ are nonnegative.

Finally, let us justify our assumption on the existence of $K$ satisfying the conditions above. Our proof will show that we can in fact take $K$ to be a closed ball. For each $x \in \overline{B}(0, 1)$, consider the closed ball $\overline{B}(rx, r)$ for some fixed radius $r > 0$. We first claim that we can take $r$ large enough so that $P(Y \in \overline{B}(rx, r)) > 0$ for all $x \in \overline{B}(0, 1)$. Suppose not. Then there exist sequences $\{x_{i}\} \subset \overline{B}(0, 1)$, $r_{i} \to \infty$, so that $P(Y \in \overline{B}(r_{i}x_{i}, r_{i})) = 0$ for all $i$. Pass to a subsequence of the $\{x_{i}\}$ converging to $x_{\infty}$. If $x_{\infty} = 0$, the liminf of the $\overline{B}(r_{i}x_{i}, r_{i})$ is all of $\RR^{n}$, hence $P(Y \in \RR^{n}) = 0$, contradiction. If $x_{\infty} \neq 0$, the liminf of the $\overline{B}(r_{i}x_{i}, r_{i})$ contains $\{y \in \RR^{n} \colon \langle x_{\infty}, y \rangle > 0\}$, hence $P(Y \in \{y \in \RR^{n} \colon \langle x_{\infty}, y \rangle > 0\}) = 0$. This implies 
$$\mathbb{E}[Y 1_{\langle Y , x_{\infty}\rangle \leq 0}] = \mathbb{E}[Y] = 0,$$ since $Y$ is centered by assumption. But this means that $Y$ is concentrated on the hyperplane $\{y \in \RR^{n} \colon \langle x_{\infty}, y \rangle = 0\}$, which also contradicts our assumptions on $Y$. So, we can indeed choose an appropriately large $r$ as claimed.

Now, note that if $ \mathbb{E}\left[Y 1_{Y \in \overline{B}(rx, r)}\right]$ is zero for some $x$, we are already done. So, let us assume for a contradiction that $\mathbb{E}\left[Y 1_{Y \in \overline{B}(rx, r)}\right] \neq 0$ for all $x \in \overline{B}(0, 1)$. We therefore obtain a continuous map  $p \colon \overline{B}(0, 1) \to S^{n-1}$ given by
$$p(x) = \frac{\mathbb{E}\left[Y 1_{Y \in \overline{B}(rx, r)}\right]}{\norm{\mathbb{E}\left[Y 1_{Y \in \overline{B}(rx, r)}\right]}}.$$

For any $x \in \partial B(0, 1) = S^{n-1}$, note that $\overline{B}(rx, r)$ is contained in the closed half-space $\{y \in \RR^{n} \colon \langle y, x \rangle \geq 0 \}$, with the only intersection with the boundary hyperplane being the origin. Therefore, we have 
$$\langle p(x), x \rangle \geq 0$$ for all $x \in \partial B(0, 1)$. It follows that the restriction $p|_{\partial B(0, 1)} \colon S^{n-1} \to S^{n-1}$ is homotopic to the identity, with an explicit homotopy given by 
$$P(x, t)  \coloneqq  \frac{(1-t) p(x) + tx}{\norm{(1-t) p(x) + tx}}.$$ But 
$P_{0}(x, t) \coloneqq p(tx)$ is a homotopy between $p|_{\partial B(0, 1)}$ and a constant map, contradiction. This completes the other direction of the second part of Proposition \ref{log-concave}. \qed

\medskip

\begin{rmk}
As a side note, the following result is a partial  converse to some of the ideas in the above proof.
\begin{prop}\label{prop: conversedomination}
    Let $X$ be a centered discrete random vector in $\R^n$ taking the values $x_1,\ldots, x_m$ with probabilities $p_1,\ldots, p_m$, and let $Y$ be another centered random vector. Define the function $g:\mathcal{S}\rightarrow \R$ as in (\ref{def of g}) and (\ref{def of S}). If $$\ds \lim_{(U,V)\in \mathcal{S},\ \norm{(U,V)}\rightarrow \infty} g(U,V)=-\infty,$$
then $X\preceq_{\mathrm{cx}} Y$. 
\end{prop}
\begin{proof}
Note that $$\ds \lim_{\lambda \rightarrow \infty}\dfrac{1}{\lambda }g(\lambda U,\lambda V)= \sum\limits_{i=1}^{m}p_i(U_i+\langle V_i,  x_i\rangle) -\E \left [\max_{1\le i\le m} (U_i+\langle V_i, Y\rangle) \right ] $$ so the right hand side must be nonpositive for any $(U,V)\in \mathcal{S}$, which implies that it must be nonpositive for all $(U,V)$. Let $f:\R^n\rightarrow \R$ be a smooth convex function. Take $$U_i=f(x_i)-\langle x_i, \nabla f(x_i)\rangle ,\ V_i=\nabla f(x_i),$$ which gives
$$\E[f(X)]=\sum\limits_{i=1}^{m}p_if(x_i)=\sum\limits_{i=1}^{m} p_i(U_i+\langle V_i, x_{i}\rangle ) \leq \E \left [ \max_{1\le i\le m}(U_i +\langle V_i, Y\rangle) \right ] \leq \E[f(Y)],$$
where the last inequality follows from convexity of $f$.
This is true for any smooth convex function $f$, so $X\preceq_{\mathrm{cx}}Y$.
\end{proof}
\end{rmk}

\section{Proof of the Convexity Conjecture} \label{section:proof convexity}

We can finally explain how to obtain Corollary \ref{cor:main}, which solves the convexity conjecture. This particular proof is due to M. Talagrand \cite[Section 2]{Tal26}. See Remark \ref{rmk: Talrem} for an alternate proof which appeared in the first version of this paper.

\begin{prop}\label{prop: hahnbanach}
Suppose $A$ is a bounded, symmetric, open subset of $\RR^{n}$, and $b, c \in (0, 1)$ satisfy $2b +c < 1$, and are such that $b^{-1}(A+A+A)$ does not contain a symmetric convex body of Gaussian measure at least $c$. Then there exists a symmetric random vector $X$ taking values in $\R^{n} \setminus b^{-1}(A+A+A)$ so that $b X \preceq_{\mathrm{cx}} G$, where $G \sim \mathcal{N}(0, I_{n})$ is a standard Gaussian random vector. 
\end{prop}
\begin{proof}

Let $\mathcal{C}$ denote the set of nonnegative, convex functions $f \colon \RR^{n} \to \RR$ satisfying $f(0) = 0$, $\mathbb{E}[f(G)] \leq 1$. For each $f \in \mathcal{C}$, let $C_{f} := f^{-1}((-\infty, b^{-1}])$, so that each $C_{f}$ is convex and satisfies $\gamma_{n}(C_{f}) \geq 1-b$.  Let $B$ be a closed ball centered at the origin of Gaussian measure at least $b + \frac{1+c}{2}$ (this is possible since $2b + c < 1$). Then for any $f \in \mathcal{C}$, $B \cap C_{f}$ is a convex body with $\gamma_{n}(B \cap C_{f}) \geq \frac{1+c}{2}$. Setting $C := B \setminus b^{-1}(A+A+A)$, we see that $C \cap C_{f}$ must be nonempty for any $f \in \mathcal{C}$, or else $b^{-1}(A +A +A)$ would contain $(B \cap C_{f}) \cap (-(B \cap C_{f}))$, which is a symmetric convex body with Gaussian measure at least $c$. 

Consider the convex subsets $\mathcal{C}_{0}, \mathcal{C}_{1}$ of the space $Z$ of continuous functions on $C$ with the sup norm,
defined as follows: $\mathcal{C}_{0}$ is simply the restriction of $\mathcal{C}$ to $C$, and $\mathcal{C}_{1}$ consists of all functions which
are strictly larger than $b^{-1}$ pointwise. Since $C \cap C_{f}$ is nonempty for any $f \in \mathcal{C}$, we see that $\mathcal{C}_{0}, \mathcal{C}_{1}$ are disjoint. Since $C$ is compact, $\mathcal{C}_{1}$ is open, hence by the geometric form of Hahn-Banach there exists a nonzero linear functional $\varphi \colon Z \to \RR$ such that 
$$\sup_{f \in \mathcal{C}_{0}} \varphi(f) \leq \inf_{f \in \mathcal{C}_{1}} \varphi(f).$$

Now observe that $\mathcal{C}_{1}$ is closed under adding an arbitrary nonnegative, continuous function $g$ defined on $C$, so in order for the previous inequality to hold we must have $\varphi(g) \geq 0$ for all such $g$. It follows that $\varphi$ is given by integration against a nontrivial nonnegative Borel measure $\mu$ supported in $C$. We can rescale $\varphi$ so that $\mu$ is a probability measure. Then, by choosing a sequence of constant functions $f_{i} \in \mathcal{C}_{1}$ with value approaching $b^{-1}$, we see that $$\sup_{f \in \mathcal{C}_{0}} \varphi(f) = \sup_{f \in C_{0}} \int f d \mu \leq b^{-1}.$$ 

Letting $Y$ be a random vector with law $\mu$, we therefore have $\mathbb{E}[f(Y)] \leq b^{-1}$ for any $f \in \mathcal{C}$. Since $\mathbb{E}[f(G)] \leq 1$ for any $f \in \mathcal{C}$, by rescaling we conclude that for any nonnegative, convex $f$ satisfying $f(0) = 0$, we have $\mathbb{E}[f(Y)] \leq b^{-1}\mathbb{E}[f(G)]$. Since $f$ is convex, we then have $\mathbb{E}[f(bY)] \leq \mathbb{E}[f(G)]$ for all nonnegative, convex $f$ satisfying $f(0) = 0$. Since these conditions on $f$ are unchanged by precomposition with $x \mapsto -x$, we conclude that $\mathbb{E}[f(-bY)] \leq \mathbb{E}[f(G)]$ for all such $f$ as well.

Take $\eta$ to be a Rademacher random variable which is independent of $Y$, so that $X = \eta Y$ is symmetric, takes values in $C$, and satisfies $\mathbb{E}[f(bX)] \leq \mathbb{E}[f(G)]$ for all nonnegative, convex $f$ with $f(0) = 0$. Since any convex $f$ can be made to satisfy these conditions by adding an affine function, and since we have equality in the previous inequality for affine functions, we therefore have $\mathbb{E}[f(bX)] \leq \mathbb{E}[f(G)]$ for all convex $f$. This shows $bX \preceq_{\mathrm{cx}} G$, as desired.
\end{proof}

\begin{cor}\label{prop:implication}
Suppose that for any $n\geq 1$, any random vector in $\R^n$ which is dominated in convex order by a standard Gaussian random vector is the sum of three standard Gaussian vectors. Let $b, c \in (0, 1)$ satisfy $2b +c \leq 1$. Then for any closed set $A \subset \R^{n}$ with $\gamma_{n}(A) > \frac{5}{6}$, the set $b^{-1}(A+A+A)$ contains a symmetric convex body $K$ with $\gamma_{n}(K) \geq c$.
\end{cor}
\begin{proof}
First, by replacing $A$ with the intersection of $A$ with a sufficently large ball centered at the origin, it is enough to show the proposition for compact $A$. Then note that if for $c \in (0, 1)$ the compact set $b^{-1}(A+A+A)$ contains a symmetric convex body $K_{c'}$ with $\gamma_{n}(K_{c'}) \geq c'$ for all $c' < c$, then by the Blaschke compactness theorem a subsequence of the $K_{c'}$ converge (in the Hausdorff sense) to a symmetric convex body $K$ with $\gamma_{n}(K) \geq c$, which is contained in $b^{-1}(A+A+A)$. So, we may further assume the strict inequality $2b +c < 1$.

Next, let us explain why it suffices to prove the proposition for bounded open sets $A$ instead of compact sets $A$.  
 Assume that the conclusion of the proposition holds for bounded open sets. Let $A$ be a compact set in $\R^n$, and for $0 < t <1$, let $A_t$ be the open $t$-neighborhood of $A$ in $\R^n$.  By assumption, if $\gamma_n(A) > \frac{5}{6}$ then $b^{-1}(A_t+A_t+A_t)$ contains a symmetric convex body $K_t$ with $\gamma_n(K_t)\geq c$. Note that the $b^{-1}(A_t+A_t+A_t)$, and therefore the $K_{t}$, are uniformly bounded. By the Blaschke compactness theorem, we can find a sequence $t_i\to 0$ such that $K_{t_i}$ converges in the Hausdorff topology to a convex set $K_0$ and $\gamma_n(K_{t_i}) \to \gamma_n(K_0)\geq 1/2$. Moreover, 
$b^{-1}(A_{t_i}+A_{t_i}+A_{t_i})$ converges in the Hausdorff topology to $b^{-1}(A+A+A)$, so $K_0\subset b^{-1}(A+A+A)$, which is what we want. Thus it is enough to show the proposition for bounded open subsets.

Suppose then that $A$ is a bounded open subset of $\R^n$ such that $\gamma_n(A) > \frac{5}{6}$. Suppose, for a contradiction, that $b^{-1}(A+A+A)$ does not contain a symmetric convex body of Gaussian measure at least $c$. Let $A' = A \cap (-A)$, so that $A'$ is bounded, symmetric, open, and satisfies $\gamma_{n}(A') > \frac{2}{3}$. Note that $b^{-1}(A'+A'+A')$ also does not contain a symmetric convex body of Gaussian measure at least $c$. 

By Proposition \ref{prop: hahnbanach}, there exists a symmetric random vector $X$ taking values in $\RR^{n} \setminus b^{-1}(A'+A'+A')$ so that $b X \preceq_{\mathrm{cx}} G$. By assumption, we may write $bX =  G_{1} + G_{2} + G_{3}$, where each $G_{i}$ is a standard Gaussian random vector. Then $\mathbb{P}(G_{i} \in A') > \frac{2}{3}$ for each $i$, and therefore by the union bound $\mathbb{P}(bX = G_{1} + G_{2} + G_{3} \in A' + A' + A') > 0$. But this implies $\mathbb{P}(X \in b^{-1}(A' + A'+A')) > 0$, contradicting that $X$ takes values in $ \RR^{n} \setminus b^{-1}(A' +A' +A')$. This completes the proof.
\end{proof}

Combining Theorem \ref{thm:main} and Corollary \ref{prop:implication} (taking $b = \frac{1}{4}, c = \frac{1}{2}$) immediately imply
Corollary \ref{cor:main} and a positive solution to Problem \ref{T}. Note that the choice of the constant $\frac{5}{6}$ is nonessential for Problem \ref{T}. Certainly having $\gamma_{n}(A) > \frac{2}{3}$ is enough, since for instance by \cite[Lemma 1.2, Lemma 1.3]{Son26} there is a universal constant $p$ so that the $p$-fold Minkowski sum of $A$ has Gaussian measure greater than $\frac{5}{6}$.

\begin{rmk}\label{rmk: Talrem}
The short ``Hahn-Banach proof'' of Corollary \ref{cor:main} from Theorem \ref{thm:main} presented above is due to Talagrand. In the first version of our paper, we showed Corollary \ref{cor:main} without explicit constants using a related but slightly less direct argument, based on the equivalence between Problem \ref{T} and Problem \ref{S} \cite[Theorem 1.1]{Son26}. That argument first uses that by van Handel's result \cite{Van25} and Theorem \ref{thm:main}, 
there is a universal $c>0$ such that for any $1$-subgaussian random vector in $\mathbb{R}^n$, $cX$ is a sum of three Gaussian vectors. Then by a refinement of one direction in the equivalence of \cite[Theorem 1.1]{Son26}, we  obtain Corollary \ref{cor:main} without explicit constants. We note that the relevant direction in \cite[Theorem 1.1]{Son26} is based on a result of Dadush-Garg-Lovett-Nikolov \cite[Theorem 1.2]{Dad19}, which is established using von Neumann’s minimax principle (see proof of \cite[Theorem 3.4]{Dad19}).
\end{rmk}

\medskip

\appendix

\section{Combinatorial Consequences}\label{app: combcons}

For the reader's convenience, we explain how the geometric convexity problem Corollary \ref{cor:main} implies its combinatorial analogue Corollary \ref{cor:comb}.
First we show that our affirmative answer to Problem \ref{T} implies an affirmative answer to the following ``solid" variant of Problem \ref{T}. We say that $K \subseteq \RR^{n}$ is solid if for any $y$ for which there exists $x \in K$ with $|x_{i}| \geq |y_{i}|$ for all $i$, we have $y \in K$.\footnote{Note that in \cite{Tal95} there is a typo in this definition.} For ease of notation let us write $A_{(q)}$ for the $q$-fold Minkowski sum of $A$ with itself.

\begin{thm}\label{thm: solid}
    There exists $\varepsilon>0$ such that for any $n \geq 1$ and any closed solid set $A$ in $\RR^{n}$ with $\gamma_{n}(A) \geq 1-\varepsilon$, there is a symmetric solid convex body $C$ in $\RR^{n}$ such that 
    $$\gamma_{n}(C) \geq \frac{1}{2} \quad \text{and}\quad  C\subset \varepsilon^{-1}A_{(3)}.$$
\end{thm}

\begin{proof}
     By Corollary \ref{cor:main}, assuming that $\varepsilon>0$ is small enough, there exists a convex body $K$ in $\RR^{n}$ such that $\gamma_{n}(K) \geq \frac{1}{2}$ and $K \subset \varepsilon^{-1}A_{(3)}.$
    Let $O$ be an orthant of $\RR^{n}$ so that $\gamma_{n}(K \cap O)$ is maximal among all orthants. Then let $C$ be the ``solidification" of $K \cap O$:
    $$C \coloneqq \{y \in \RR^{n} \mid \exists \; x \in K \cap O, |x_{i}| \geq |y_{i}| \; \text{for all } i\}.$$ By definition, $C$ is solid, and thus symmetric. 
    We want to show that $C$ is convex. Let $y^{(0)}, y^{(1)} \in C$ be arbitrary, and define $y^{(t)} \coloneqq (1-t)y^{(0)} + ty^{(1)}$ for $t \in [0, 1]$. By definition, there exist $x^{(0)}, x^{(1)} \in K \cap O$ with $|x^{(0)}_{i}| \geq |y^{(0)}_{i}|, |x^{(1)}_{i}| \geq |y^{(1)}_{i}|$ for all $i$. For each $t$, by convexity we have $x^{(t)} \in K \cap O$, where $x^{(t)} \coloneqq (1-t)x^{(0)} + tx^{(1)}$. Note that since $x^{(0)}, x^{(1)} \in O$, $x^{(0)}_{i}, x^{(1)}_{i}$ have the same sign for each $i$. It follows that 
    $$|x^{(t)}_{i}| = (1-t)|x^{(0)}_{i}| + t |x_{i}^{(1)}| \geq (1-t)|y_{i}^{(0)}| + t |y_{i}^{(1)}| \geq |y_{i}^{(t)}|$$ for each $i$. Therefore $y^{(t)} \in C$, and so $C$ is convex as claimed.

    By the definition of solidity, $C$ is closed under inverting the sign of any coordinate. The copies of $K \cap O$ obtained by inverting a subset of the coordinates are all mutually disjoint (up to $\gamma_{n}$-null sets), so 
    $$\gamma_{n}(C) \geq 2^{n}\gamma_{n}(K \cap O) \geq \gamma_{n}(K)\geq \frac{1}{2},$$ where the second inequality follows from the maximality of $O$. 

    It remains to check that $$C \subset \varepsilon^{-1}A_{(3)}.$$ 
    Since $A$ is solid, $\eps^{-1}A_{(3)}$ is solid too. Moreover $\eps^{-1}A_{(3)}$ contains $K \cap O$, so it must contain the solidification $C$, and we are done.
\end{proof}

\begin{rmk}\label{rmk: strictlygreater}
By a standard argument, in Theorem \ref{thm: solid} we may in fact guarantee $\gamma_{n}(C) \geq \frac{2}{3} + \delta$ for some small $\delta > 0$, at the cost of decreasing the universal constant $\varepsilon$. This follows from the elementary fact that if $C$ is symmetric convex and  $\gamma_{n}(C) \geq \frac{1}{2}$, then for some universal $\lambda>0$, $\gamma_{n}(\lambda C) \geq \frac{2}{3} + \delta$. Therefore, by taking the intersection with a sufficiently large ball, we can guarantee that $C$ is both compact and satisfies $\gamma_{n}(C) \geq \frac{2}{3}$. Also, note that the assumption that $A$ is closed is unnecessary at the cost of modifying $\eps$ slightly, since we can always approximate $A$ by a closed solid subset with arbitrarily small measure loss.
\end{rmk}

\medskip 

Note that Theorem \ref{thm: solid} is a slightly weaker version of the still open \cite[Problem 3.1]{Tal95}, with $\varepsilon^{-1}(A+A)$ replaced by $\varepsilon^{-1}(A+A+A)$. We claim that this version implies the following modified version of \cite[Problem 3.2]{Tal95}.
\begin{thm}\label{thm: logsequences}
    There exist constants $\eps, L >0$ so that for each solid subset $A \subset \RR^{n}$ with $\gamma_{n}(A) \geq 1-\eps$, there exists a finite sequence $\{I_{k}\}_{k \geq 1}$ of subsets of $[n]$ such that 
    \begin{itemize}
        \item $L |I_{k}| \geq \log(k+1)$ for each $k,$
        \item For any $x \in \RR^{n}$, if $\sum_{i \in I_{k}} x_{i}^{2} \leq L^{-1}|I_{k}|$  for all $k \geq 1$, then $x \in A_{(3)}$.
    \end{itemize}
\end{thm}

The proof of this implication relies on the following result. 
\begin{prop}[{\cite[Theorem 19.2.15, Exercise 19.2.16]{Tal21}}]\label{prop: unconditional}
    There exists $L > 0$ so that the following is true. Let $\{e_{i}\}_{i\leq n}$ be a $1$-unconditional sequence in a Banach space $(X, \norm{\cdot})$, and let $S = \mathbb{E}\left[\norm{\sum_{i=1}^{n}g_{i}e_{i}} \right],$ where $\{g_{i}\}_{i \leq n}$ are independent standard Gaussian random variables. Then we can find a finite sequence $\{I_{k}\}_{k \geq 1}$ of subsets of $[n]$ such that 
    \begin{itemize}
        \item $|I_{k}| \geq \log(k+1)$ for each $k$,
        \item For any $x \in X$ of the form $x = \sum_{i \leq n} x_{i}e_{i}$, we have 
        $$\norm{x} \leq LS\sup_{k \geq 1} \left(\frac{1}{|I_{k}|} \sum_{i \in I_{k}} x_{i}^{2}\right)^{\frac{1}{2}}.$$
    \end{itemize}
\end{prop}
Here $\{e_{i}\}_{i \leq n}$ being a $1$-unconditional sequence means that for any $\{a_{i}\}_{i \leq n}$, $\{\eps_{i}\}_{i \leq n}$ where each $\eps_{i}$ is either $1$ or $-1$, we have 
$$\norm{\sum_{i\leq n} \eps_{i}a_{i} e_{i}} = \norm{\sum_{i \leq n} a_{i}e_{i}}$$ (see \cite[Definition 19.2.12]{Tal21}). 
Recall the standard fact that if $K$ is a symmetric, compact, convex body containing $0$ in $\RR^{n}$, then $K$ defines a Banach norm on $\RR^{n}$ by
$$\norm{x}_{K} \coloneqq \inf\{\lambda \mid x \in \lambda K\}.$$ In our proof, we will produce such a $K$ by modifying the convex body $C$ obtained in Theorem \ref{thm: solid}.
\begin{proof}[Proof of Theorem \ref{thm: logsequences}]
    Let $\varepsilon>0$ be sufficiently small so that Theorem \ref{thm: solid} and the improvement in Remark \ref{rmk: strictlygreater} apply. Then there exists a solid, compact, convex body $C$ with $\gamma_{n}(C) \geq \frac{2}{3}$, $C \subset \varepsilon^{-1}A_{(3)}$. 
    We take $\{e_{i}\}_{i \leq n}$ to be the standard basis for $\RR^{n}$. Since $C$ is solid (and any rescaling of $C$ is solid), it is straightforward to see that $\{e_{i}\}_{i \leq n}$ is a $1$-unconditional sequence with respect to $\norm{\cdot}_{C}$.
    Therefore we may extract a sequence  $\{I_{k}\}_{k \geq 1}$ of subsets of $[n]$ satisfying the two properties given in Proposition \ref{prop: unconditional} with constant $L$. This sequence satisfies the first desired property in Theorem \ref{thm: logsequences}, so it remains to verify the second for the constant $\eps^{-2}L^{2}S^{2}$.

    Suppose $x \in \RR^{n}$ satisfies $\sum_{i \in I_{k}}x_{i}^{2} \leq \eps^{2} L^{-2}S^{-2}|I_{k}|$ for all $k$. By the second property in Proposition \ref{prop: unconditional}, it follows that 
    \begin{align*}
        \norm{x}_{C} &\leq LS\sup_{k \geq 1} \left(\frac{1}{|I_{k}|} \sum_{i \in I_{k}} x_{i}^{2}\right)^{\frac{1}{2}} \leq \eps.
    \end{align*}
By definition of $\norm{\cdot}_{C}$, we therefore have $x \in \eps C \subset A_{(3)}$.
So, the only remaining step is to verify that $S$ is bounded above by a constant independent of the dimension $n$. 
Since $\{e_{i}\}_{i \leq n}$ is just the standard basis, we have $$S = \mathbb{E}\left[\norm{G}_{C}\right]$$ where $G$ is a standard Gaussian random vector. Then $\mathbb{P}[\|G\|_C \le 1] \ge 2/3$ and so the  median of $\|G\|_C$ is at most $1$, thus a uniform bound on $S$ follows from Gaussian concentration \cite[Lemma 3.1]{LT91}. Another elementary way to check it is as follows.
We have 
$$\mathbb{P}[\norm{G}_{C} \geq t ] = 1- \gamma_{n}(tC),$$ since $tC$ is the closed ball of radius $t$ with respect to $\norm{\cdot}_{C}$. Applying a general form of the Ehrhard-Borell inequality \cite[Theorem 1.2]{Bor07}, we conclude that for $t \geq 1$, 
\begin{align}\label{eb}
\Phi_{1}^{-1}(\gamma_{n}(tC)) = \Phi_{1}^{-1}\left(\gamma_{n}\left(\frac{t}{2}C + \frac{t}{2}C\right)\right) \geq t\Phi_{1}^{-1}(\gamma_{n}(C)),
\end{align}
where $\Phi_{1}(a) = \gamma_{1}((-\infty, a])$. By assumption, we have $\gamma_{n}(C) \geq \frac{2}{3}$, so $\lambda \coloneqq \Phi_{1}^{-1}(\gamma_{n}(C))$ is a strictly positive constant. 
Now we write
\begin{align*}
    \mathbb{E}[\norm{G}_{C}] = \int_{0}^{\infty} \mathbb{P}[\norm{G}_{C} \geq t ] dt
    =  \int_{0}^{1}1- \gamma_{n}(tC)dt + \int_{1}^{\infty}1 - \gamma_{n}(tC)dt.
    \end{align*}
The first integral is trivially bounded above by $1$, whereas the second integral may be bounded by 
    \begin{align*}
     \int_{1}^{\infty}1 - \gamma_{n}(tC)dt &\leq \int_{1}^{\infty}1- \Phi_{1}(t\lambda) dt\\
    &\leq \int_{0}^{\infty} 1- \Phi_{1}(t\lambda) dt\\
    &= \int_{0}^{\infty} \mathbb{P}(\frac{1}{\lambda}g \geq t)dt\\
    &= \frac{1}{2}\mathbb{E}\left[\frac{1}{\lambda }|g|\right],
\end{align*}
where $g$ is a standard Gaussian random variable. In the first inequality we have used the monotonicity of $\Phi_{1}$ and (\ref{eb}). Therefore $S$ is bounded above by the uniform constant $1 + \frac{1}{2}\mathbb{E}\left[\frac{1}{\lambda }|g|\right]$, and we are done.
\end{proof}

Finally, as explained in the comment after \cite[Problem 3.4]{Tal95} (with slightly modified statements), Theorem \ref{thm: logsequences} in turn implies Corollary \ref{cor:comb}.
\begin{proof}[Proof of Corollary \ref{cor:comb}]
Let $t>0$ be such that $p=\gamma_1(|x|\geq t)$ and consider the map $F  \colon \R^N\to \{0,1\}^N$ given by $F((x_1,...,x_N)) := (1_{|x_1|\geq t},...,1_{|x_N|\geq t})$.  Let us assume that $A$ is decreasing, namely $y\in A$ whenever $y\subset x$ for some $x\in A$.
This can always be assumed when proving Corollary \ref{cor:comb} because if $A$ is not decreasing and if $A_0$ is the smallest decreasing subset of $\{0,1\}^N$ containing $A$, then $A^{(q)} = A_0^{(q)}$. 
Let $\varepsilon,L>0$ be the constants  given by Theorem \ref{thm: logsequences}, and suppose that 
$\mu_p(A)\geq  1-\varepsilon$. Let $\{I_k\}_{k\geq 1}$ be the subsets of $[N]$ given by Theorem \ref{thm: logsequences}. Note that  $\mu_p(A) = \gamma_N(F^{-1}(A))$ and that $F^{-1}(A)$ is solid since $A$ is decreasing.
Consider the family $\mathcal{J}=\cup_{k} \mathcal{J}_k$ where $\mathcal{J}_k$  is the family of subsets of $I_k$ with cardinality $\left\lceil \frac{|I_k|}{9t^2L} \right\rceil$. Let us prove that $A^{(3)}$ is $p^{L'}$-small for some universal $L'\geq 1$.  

First, we need to show that $A^{(3)} \subset \bigcup_{J\in \mathcal{J}}H_J$ where $H_J:= \{J' \subset [N] \,|\,J\subset J'\}$. Consider $y\in \{0,1\}^N$ and assume that for all $J\in \mathcal{J}$, there is an index $i$ in $J$ so that $y_i=0$. Set $x=3ty$.
Then it is enough to check that \begin{equation}\label{f-1a}
x\in F^{-1}(A)+F^{-1}(A)+F^{-1}(A)
\end{equation}
since then, $x=a+b+c$ where $a,b, c \in F^{-1}(A)$ and when $y_i=1$, then either $a_i\geq t$ or $b_i\geq t$ or $c_i\geq t$ so that either $F(a)_i=1$ or $F(b)_i=1$ or $F(c)_i=1$, which means $A^{(3)} \subset \bigcup_{J\in \mathcal{J}}H_J$ as wanted. Suppose that  (\ref{f-1a}) were not true. Then by  Theorem \ref{thm: logsequences} we have
$$|\{i\in I_k\,| \, y_i=1\}| \geq \left\lceil \frac{|I_k|}{9t^2L} \right\rceil$$
for some $k$ (note that the left hand side is an integer), and so we can find $J\in \mathcal{J}$ with $y_i=1$ for all $i\in J$, contradiction. 

Second, we need to prove that 
\begin{equation}\label{sumJJ}
\sum_{J\in \mathcal{J}} \mu_{p^{L'}}(H_J) = \sum_{J\in \mathcal{J}} p^{L'|J|} \leq 1/2
\end{equation}
for some universal $L'\geq 1$. Given a parameter $M\geq 1$, we have for each $k$:
$$\sum_{J\in \mathcal{J}_k} p^{M|J|} = p^{M \left\lceil\frac{|I_k|}{9t^2L} \right\rceil} \binom{|I_k|}{\left\lceil\frac{|I_k|}{9t^2L} \right\rceil} \leq p^{M\left\lceil\frac{|I_k|}{9t^2L} \right\rceil} (9et^2L)^{\left\lceil\frac{|I_k|}{9t^2L} \right\rceil}.$$ Using that $p=\gamma_1(\{|x|\geq t\}) \leq e^{-t^2/2}$, that $M \left \lceil \frac{|I_{k}|}{9t^{2}L} \right \rceil \geq M \frac{|I_{k}|}{18t^{2}L} + \frac{M}{4}$, and that $(ez)^{1/z}\leq e$ for $z = 9t^2L$, we obtain that the previous expression is at most
$$\exp\left(-|I_k|\left(\frac{M}{36L} -1\right)\right) \cdot 9e^{1- \frac{Mt^{2}}{8}}t^2L \leq (k+1)^{-(\frac{M}{36L} -1)/L} \cdot 9e^{1- \frac{Mt^{2}}{8}}t^2L $$
since $L|I_k|\geq \log(k+1)$. Then (\ref{sumJJ}) follows by choosing $M$ large enough with respect to $L$, as well as choosing $M > 8$ so that $e^{1- \frac{Mt^{2}}{8}} t^{2} \leq 1$, and finally taking $L':=M$.

\end{proof}

\section{Coupling from Laguerre tessellations}\label{section:laguerre}

A statement weaker than Proposition \ref{log-concave} is enough to prove Theorem \ref{thm:main}: one only needs to find a martingale coupling between $X$ and the standard Gaussian $G$ such that $G$ conditioned on $X=x_j$ has a $1$-uniform log-concave distribution for any $x_j$. 
This fact can be proved using Proposition \ref{prop:tessellation} below, which was generated by ChatGPT-5.5 Pro. 
The results presented in Section \ref{entropy} are more general as explained at the end of this section.

Before stating the proposition, let us review some definitions.
We say $R \subset \RR^n$ is a convex polyhedral subset if it has nonempty interior and is the intersection of a finite number of affine half-spaces. A collection of convex polyhedral subsets $\{R_{j}\}_{j=1}^{J}$ of $\R^{n}$ is a convex polyhedral partition if their union is $\R^{n}$, and the intersections $R_{j} \cap R_{j'}$ for $j \neq j'$ have $\gamma_{n}$-measure zero.

\begin{prop}\label{prop:tessellation}
Let 
$X\sim  \sum_{i=1}^{m} p_{i}\delta_{x_{i}}$ be a centered finitely supported random vector in  $\RR^{n}$. 
Suppose that for some $0 < \rho < 1$ and some $\hat{G}\sim \mathcal{N}(0,I_n)$,  
$$X \preceq_{\mathrm{cx}} \rho \hat{G}.$$
Then there exists a convex polyhedral partition $\{R_{1}, \ldots, R_{m}\}$ of  $\RR^{n+m}$ with non-empty interiors, so that for 
$$G = (Z, \eta_{1}, \ldots, \eta_{m}) \sim \mathcal{N}(0,I_{n+m}), \quad Z \in \RR^{n},\,\, \eta_{i} \in \RR$$ we have 
$$\gamma_{n+m}(R_{i}) = p_{i} \quad \text{and}\quad \mathbb{E}[Z \mid G \in R_{i}] = x_{i}.$$
\end{prop}

Now if $\Pi:\R^{n+m}\to \R^n$ denotes the standard projection, then $\Pi(G)\sim \mathcal{N}(0,I_n)$ and Proposition \ref{prop:tessellation} produces a martingale coupling $(X,\Pi(G))$ such that $\Pi(G)$ conditioned on $X=x_j$ has probability distribution proportional to the pushforward measure $\Pi_*(\gamma_{n+m}\big|_{R_j})$, where  $\gamma_{n+m}\big|_{R_j}$ denotes the restriction of $\gamma_{n+m}$ to $R_j$. Since $R_j$ is a convex domain and $\gamma_{n+m}$ is 1-uniformly log-concave, $\Pi_*(\gamma_{n+m}\big|_{R_j})$ is  1-uniformly log-concave as well, which is what we need for proving Theorem \ref{thm:main}.

The usefulness of a result like Proposition \ref{prop:tessellation} 
to prove Theorem \ref{thm:main} was explicitly highlighted in the first version of \cite[Subsection 2.4.3]{Son26}: in the language of \cite[Subsection 2.4.3]{Son26}, Proposition \ref{prop:tessellation} confirms that for $\kappa>0$ small enough, any finitely supported $\kappa$-subgaussian random vector in $\R^n$ is a ``simple random vector''.

In hindsight, Proposition \ref{prop:tessellation} follows from the study of Laguerre tessellations (also called power diagrams), and in particular from \cite[Theorem 3.5]{inverselaguerre}, which gives a complete description of the (relative) boundary of the convex set 
$$\mathcal{C}(p_{1}, \ldots, p_{m}, \gamma_{n}) \coloneqq \left\{(x_{1}, \ldots, x_{m}) \in (\RR^{n})^{m} \colon \sum_{i=1}^{m} p_{i} \delta_{x_{i}} \preceq_{\mathrm{cx}} \gamma_{n}\right\}.$$ We remark that a similar, more general set is studied in \cite{KMSW24}, although the focus of this latter paper is specifically on the extreme points of that set.

We briefly sketch the argument for deriving Proposition \ref{prop:tessellation} using \cite[Theorem 3.5]{inverselaguerre}: fix real numbers $a_{1}, \ldots, a_{m}$ (not all zero) satisfying $\sum_{i=1}^{m}p_{i}a_{i} = 0$, and consider the set of all $\lambda $ for which $((x_{1}, \lambda a_{1}), \ldots, (x_{m}, \lambda a_{m}))$ is in $\mathcal{C}(p_{1}, \ldots, p_{m}, \gamma_{n+1})$. The supremum $\lambda^{*}$ of this set has the property that $\sum_{i=1}^{m}p_{i}\delta_{(x_{i}, \lambda^{*}a_{i})}$ is in the boundary of $\mathcal{C}(p_{1}, \ldots, p_{m}, \gamma_{n+1})$. Now, \cite[Theorem 3.5]{inverselaguerre} tells us that we have a convex polyhedral partition $L_{1}, \ldots, L_{m'}$ of $\RR^{n+1}$ and a surjective map $\sigma \colon \{1, \ldots, m\} \to \{1, \ldots, m'\}$ (where $m' \geq 2$) so that for each $j \in \{1, \ldots, m'\}$, 
$$\sum_{i \in \sigma^{-1}(j)} p_{i}\delta_{(x_{i}, \lambda^{*}a_{i})} \preceq_{\mathrm{cx}} \gamma_{n+1}|_{L_{j}}.\footnote{We are guaranteed the additional property that $\gamma_{n+1}(L_{j}) = \sum_{i \in \sigma^{-1}(j)} p_{i}$, so we can rescale both to be probability measures.}$$ If we are in the boundary again (of the convex set $\mathcal{C}(p_{1}, \ldots, p_{m},  \gamma_{n+1}|_{L_{j}})$), we may immediately apply \cite[Theorem 3.5]{inverselaguerre} once more; if not, we can add a new dimension as above and maximize a new parameter $\lambda'$ to reach the boundary, and then proceed. Note that along this inductive process, the maximum number of points corresponding to the same convex region strictly decreases, so after adding at most $m$ new dimensions we obtain the desired result (since if $p_{i}\delta_{x_{i}} \preceq_{\mathrm{cx}} \mu|_{R_{i}}$, then $R_{i}$ has $\mu$-volume $p_{i}$ and barycenter $x_{i}$).

One minor difference between this approach and the proof generated by ChatGPT, albeit immaterial for the applications to the rest of the paper, is the fact that ChatGPT's proof uses the fact that $X \preceq_{\mathrm{cx}} \rho \hat{G}$ for some $0 < \rho < 1$, whereas our approach holds as long as $X \preceq_{\mathrm{cx}} \hat{G}$. This is because if $X  \not \preceq_{\mathrm{cx}} \rho \hat{G}$ for any $0 < \rho < 1$, then the support $(x_{1}, \ldots, x_{m})$ of $X$ must be in the boundary of $\mathcal{C}(p_{1}, \ldots, p_{m}, \gamma_{n})$, at which point we may apply \cite[Theorem 3.5]{inverselaguerre} immediately. Another difference is the nature of the convex polyhedral partition $\{R_{1}, \ldots, R_{m}\}$ of $\mathbb{R}^{n+m}$. The proof given by ChatGPT shows that we can take these regions to be the cells of a Laguerre tessellation, which is a convex polyhedral partition defined by a collection of affine functions $g_{1}, \ldots, g_{m} \colon \RR^{n+m} \to \RR$ such that 
$$R_{i} = \{x \colon g_{i}(x) \geq g_{j}(x) \; \forall j\}.$$ The above proof using \cite[Theorem 3.5]{inverselaguerre} only guarantees a hierarchical Laguerre tessellation, namely a convex polyhedral partition generated by iteratively subdividing (via further Laguerre tesselations) the cells of a Laguerre tessellation.

As with Proposition 3.6, both ChatGPT's proof and our sketched proof of Proposition \ref{prop:tessellation} hold with $\gamma_{n}$ replaced by a more general measure $\mu$, as long as $\mu$ satisfies standard assumptions such as being absolutely continuous with respect to the Lebesgue measure and having finite first moments. 

Finally, we briefly sketch how to recover the characterization of the relative boundary in \cite[Theorem 3.5]{inverselaguerre} (item 3 of that proposition) from Proposition \ref{log-concave}. Along with the previous discussion, this shows how one can obtain Proposition \ref{prop:tessellation} from Proposition \ref{log-concave}. For a probability measure $\mu'$ on $\mathbb{R}^n$, having support being in the relative interior of $\mathcal{C}(p_{1}, \ldots, p_{m}, \nu))$ is equivalent to the ``slack'' condition that $X'\preceq_{\mathrm{cx}} (1-\varepsilon)Y$ for some $\varepsilon>0$ and $X'\sim \mu'$, $Y\sim \nu$, 
since $(0, \ldots, 0)$ is in the relative interior of $\mathcal{C}(p_{1}, \ldots, p_{m}, \nu)).$
Now, let $\mu = \sum_{i=1}^{m} p_{i}\delta_{x_{i}}$ be a centered measure whose support is in the relative boundary of $\mathcal{C}(p_{1}, \ldots, p_{m}, \nu)$, where $\nu$ is a centered probability measure satisfying the assumptions of Proposition \ref{log-concave}. Further suppose that $\{\mu_{k}\}_{k=1}^{\infty}$ is a sequence of measures whose supports are in the relative interior of $\mathcal{C}(p_{1}, \ldots, p_{m}, \nu))$, and converge to the support of $\mu$. Then for each $k$ we have functions $f_{1k}, \ldots, f_{mk} \colon \RR^{n} \to \RR$ of the form
$$f_{ik}(x) = \frac{e^{U_{ik} + \langle V_{ik}, x \rangle}}{\sum_{\ell=1}^{m}p_{\ell} e^{U_{\ell k} + \langle V_{\ell k}, x \rangle}}$$ satisfying the properties guaranteed by the second part of Proposition \ref{log-concave}. 
Note that by dividing each numerator and denominator by $e^{U_{1k} + \langle V_{1k}, x \rangle}$, we may assume without loss of generality that $(U_{1k}, V_{1k}) = 0$ for all $k$. 

Let $N_{k} = \max_{1 \leq i \leq m} \norm{(U_{ik}, V_{ik})}$; we claim that $N_{k} \to \infty$. Suppose, for a contradiction, that this were not the case. Then by compactness we may assume, by passing to a subsequence, that each $(U_{ik}, V_{ik})$ converges to some $(U_{i}, V_{i})$. Taking the limits of the corresponding $f_{ik}$ gives nonnegative functions $f_{1}, \ldots, f_{m}$ of the form
$$f_{i}(x) = \frac{e^{U_{i} + \langle V_{i}, x \rangle}}{\sum_{\ell=1}^{m}p_{\ell} e^{U_{\ell } + \langle V_{\ell }, x \rangle}}.$$ These then constitute conditional densities giving a martingale coupling of $X \sim \mu$ and $Y \sim \nu$, which by the second part of Proposition \ref{log-concave} implies $X \preceq_{\mathrm{cx}} (1-\eps)Y$ for some $\eps > 0$, contradicting that $\mu$ is on the relative boundary of $\mathcal{C}(p_{1}, \ldots, p_{m}, \nu)$. So indeed $N_{k} \to \infty$.

Now let $(\widehat{U}_{ik}, \widehat{V}_{ik}) = \frac{1}{N_{k}}(U_{ik}, V_{ik})$; by compactness we may pass to a subsequence and obtain limits $(\widehat{U}_{ik}, \widehat{V}_{ik}) \to (\widehat{U}_{i}, \widehat{V}_{i})$. This then gives a surjective function $\sigma \colon \{1, \ldots, m\} \to \{1, \ldots, m'\}$, for some $m' \leq m$, defined so that $\sigma(i) = \sigma(i')$ iff $(\widehat{U}_{i}, \widehat{V}_{i}) = (\widehat{U}_{i'}, \widehat{V}_{i'})$. Note that, as in \cite[Theorem 3.5]{inverselaguerre}, we must have $m' \geq 2$, since at least one $(\widehat{U}_{i}, \widehat{V}_{i})$ must have strictly positive norm whereas $(\widehat{U}_{1}, \widehat{V}_{1}) = 0$ by construction. 

For each $j \in \{1, \ldots, m'\}$, we therefore have an affine function $g_{j} \colon \RR^{n} \to \RR$ given by $g_{j}(x) = \widehat{U}_{i} + \langle \widehat{V}_{i}, x \rangle$ for any $i \in \sigma^{-1}(j)$, along with the associated Laguerre cell
$$R_{j} =  \{x \colon g_{j}(x) \geq g_{j'}(x) \; \forall j' \in \{1, \ldots, m'\}\}.$$ This gives a Laguerre tessellation of $\RR^{n}$, which we claim satisfies the conditions of the third item in \cite[Theorem 3.5]{inverselaguerre}. To see this, note that for any $j \neq j' \in \{1, \ldots, m'\}$, any $x \in R_{j} \setminus R_{j'}$, and any any $i, i'$ with $\sigma(i) = j, \sigma(i') = j'$, we have
\begin{align*}
    \lim_{k \to \infty} f_{i'k}(x) &= \lim_{k \to \infty} \frac{1}{\sum_{\ell = 1}^{m} p_{\ell} e^{U_{\ell k} + \langle V_{\ell k}, x \rangle - U_{i'k} - \langle V_{i' k}, x \rangle}}\\
    &= \lim_{k \to \infty} \frac{1}{\sum_{\ell =1}^{m} p_{\ell}\left(e^{\widehat{U}_{\ell k} + \langle \widehat{V}_{\ell k}, x \rangle - \widehat{U}_{i'k} - \langle \widehat{V}_{i' k}, x \rangle} \right)^{N_{k}}}\\
    &= 0,
\end{align*}
where the last equality follows since $$\lim_{k \to \infty} e^{\widehat{U}_{i k} + \langle \widehat{V}_{i k}, x \rangle - \widehat{U}_{i'k} - \langle \widehat{V}_{i' k}, x \rangle} = e^{g_{j}(x) - g_{j'}(x)} > 1,$$ and $N_{k} \to \infty$. So, the only conditional densities which are possibly nonzero in $R_{j}$ (outside of the $\nu$-null intersections $R_{j} \cap R_{j'}$) are the limiting densities $f_{i}$ for which $i \in \sigma^{-1}(j)$, and so we have 
$$\sum_{i \in \sigma^{-1}(j)} p_{i} \delta_{x_{i}} \preceq_{\mathrm{cx}} \nu|_{R_{j}}$$ as desired.

\bibliographystyle{alpha}
\bibliography{main}

\end{document}